\documentclass[twocolumn, 10pt]{amsart}

 \calclayout

\usepackage{amsmath,amssymb,color}
\usepackage[dvipdfmx]{graphicx}
\usepackage[abs]{overpic} 
\usepackage[dvipdfmx]{pict2e}
\usepackage{lscape}

\newtheorem{theo}{Theorem}

\newtheorem{lem}[theo]{Lemma}

\DeclareMathOperator{\R}{\mathbb{R}}
\DeclareMathOperator{\Z}{\mathbb{Z}}

\begin{document}
\twocolumn[\begin{center}

\textbf{{\large Root systems and graph associahedra}}\\
\vspace{0.5cm}
Miho Hatanaka\\
Department of Mathematics, Osaka City University, Sumiyoshi-ku, Osaka 558-8585, Japan.\\
e-mail : hatanaka.m.123@gmail.com\\
\end{center}
\date{\today}
\vspace{0.5cm}
\textbf{Abstract} : 
It is known that a connected simple graph $G$ associates a simple polytope $P_G$ called a graph associahedron in Euclidean space.  In this paper we show that the set of facet vectors of $P_G$ forms a root system if and only if $G$ is a cycle graph and that the root system is of type A.\\
\vspace{0.5cm}
\textbf{1.\ Introduction.}\ \ \ 
Let $G$ be a connected simple graph with $n+1$ nodes and its node set $V(G)$ be $[n+1]=\{1,2, \dots, n+1 \}$.
We can construct the graph associahedron $P_G$ in ${\R}^n$ from $G$ (\cite{po}). 
We call a primitive (inward) normal vector to a facet of $P_G$ a \emph{facet vector} and denote by $F(G)$ the set of facet vectors of $P_G$.  
One can observe that when $G$ is a complete graph, $F(G)$ agrees with the primitive edge vectors of the fan formed by the Weyl chambers of a root system of type A (\cite{ab}), in other words, $F(G)$ is \emph{dual} to a root system of type A when $G$ is a complete graph.  
Motivated by this observation, we ask whether $F(G)$ itself forms a root system for a connected simple graph $G$.  
It turns out that $F(G)$ forms a root system if and only if $G$ is a cycle graph (Theorem~\ref{theo:4-1}). 
On the way to prove it, we show that $F(G)$ is centrally symmetric (this is the case when $F(G)$ forms a root system) if and only if $G$ is a cycle graph or a complete graph.  

\smallskip
\textbf{2.\ Construction of graph associahedra.}\ \ \ 
We set 
$$B(G) := \{I \subset V(G) \mid G|I\ \mathrm{is\ connected} \},$$
where $G|I$ is a maximal subgraph of $G$ with the node set $I$ (i.e. the induced subgraph).
The empty set $\emptyset$ is not in $B(G)$.
We call $B(G)$ a {\it graphical building set} of $G$.
We take an $n$-simplex in ${\R}^n$ such that its facet vectors are $e_1, \dots, e_n,$ and $-e_1- \dots -e_n$, where $e_1, \dots, e_n$ are the standard basis of ${\R}^n$.
Each facet vector $e_i\ (1 \leq i \leq n)$ corresponds to an element $\{i\}$ in $B(G)$, and the facet vector $-e_1- \dots -e_n$ corresponds to an element $\{n+1\}$ in $B(G)$. 
We truncate the $n$-simplex along faces in increasing order of dimension.
Let $F_i$ denote the facet of the simplex corresponding to $\{i\}$ in $B(G)$.
For every element 
$I=\{i_1,\dots,i_k\}$ in $B(G) \setminus [n+1]$ we truncate the simplex along the face $F_{i_1} \cap \dots \cap F_{i_k}$ in such a way that the facet vector of the new facet, denoted $F_I$, is the sum of the facet vectors of the facets $F_{i_1}, \dots, F_{i_k}$.
Then the resulting polytope, denoted $P_G$, is called a {\it graph associahedron}.
We denote by $F(G)$ the set of facet vectors of $P_G$.  

\smallskip
\textbf{3.\ Facet vectors associated to complete grap-hs.}\ \ \ 
As mentioned in the Introduction, $F(G)$ is \emph{dual} to a root system of type A when $G$ is a complete graph.  We shall explain what this means.  If $G$ is a complete graph $K_{n+1}$ with $n+1$ nodes, then the graphical building set $B(K_{n+1})$ consists of all subsets of $[n+1]$ except for $\emptyset$ so that the graph associahedron $P_{K_{n+1}}$ is a permutohedron obtained by cutting all faces of the $n$-simplex with facet vectors $e_1,\dots, e_n, -(e_1+\dots+e_n)$.  It follows that  
\begin{equation} \label{eq:4-1}
F(K_{n+1})=\big\{ \pm \sum_{i\in I}e_i\mid \emptyset\not=I\subset [n]\big\}.
\end{equation}
On the other hand, consider the standard root system $\Delta(A_n)$ of type $A_n$ given by 
\begin{equation} \label{eq:4-1-1}
\Delta(A_n)=\{ \pm(e_i-e_j)\mid 1\le i<j\le n+1\} 
\end{equation} 
which lies on the hyperplane $H$ of $\R^{n+1}$ with $e_1+\dots+e_{n+1}$ as a normal vector.  
Take $e_1-e_2, e_2-e_3,\dots, e_n-e_{n+1}$ as a base of $\Delta(A_n)$ as usual.  
Then their dual base with respect to the standard inner product on ${\R}^{n+1}$ is what is called the fundamental dominant weights given by 
\begin{eqnarray*} \label{eq:4-1-2}
&\lambda_i=(e_1+\dots+e_i)-\frac{i}{n+1}(e_1+\dots+e_{n+1}) \nonumber \\
&\text{$(i=1,2,\dots,n)$} 
\end{eqnarray*}
which also lie on the hyperplane $H$.  
The Weyl group action permutes $e_1,\dots,e_{n+1}$ so that it preserves $H$.  
We identify $H$ with the quotient vector space $H^*$ of $\R^{n+1}$ by the line spanned by $e_1+\dots+e_{n+1}$ using the inner product, namely put the condition $e_1+\dots+e_{n+1}=0$. 
Then the set of elements obtained from the orbits of $\lambda_1,\dots,\lambda_n$ by the Weyl group action is  
\[
\big\{\sum_{j\in J}e_j\mid \emptyset\not=J\subset [n+1]\big\} \quad \text{in $H^*$}.
\]
This set agrees with $F(K_{n+1})$ in \eqref{eq:4-1} because $e_{n+1}=-(e_1+\dots+e_n)$. In this sense $F(K_{n+1})$ is \emph{dual} to $\Delta(A_n)$. 

\smallskip
\textbf{4.\ Main theorem.}\ \ \ 
We note that $F(K_{n+1})$ itself forms a root system (of type $A_n$) when $n=1$ or $2$. However the following holds. 

\noindent
\begin{lem} \label{lem:4.1.0}
 {\it If $n\ge 3$, then $F(K_{n+1})$ does not form a root system.}
\end{lem}

\begin{proof}
Suppose that $F(K_{n+1})$ forms a root system for $n\ge 3$.  Then $F(K_{n+1})$ is of rank $n$ and the number of positive roots is $2^n-1$ by \eqref{eq:4-1}.  On the other hand, no irreducible root system of rank $n(\ge 3)$ has $2^n-1$ positive roots (see \cite[Table 1 in p.66]{hu}).  Therefore, it suffices to show that $F(K_{n+1})$ is irreducible if it forms a root system. 

Let $V$ be an $m$-dimensional linear subspace of $\R^{n}$ such that $E=F(K_{n+1})\cap V$ is a root subsystem of $F(K_{n+1})$.  We consider the mod $2$ reduction map 
\[
\varphi \colon \Z^n\cap V\to (\Z^n\cap V)\otimes \Z/2
\]
where $\Z/2=\{0,1\}$. Since $(\Z^n\cap V)\otimes \Z/2$ is a vector space over $\Z/2$ of dimension $\le m$, it contains at most $2^m-1$ nonzero elements.  On the other hand, since the coordinates of an element in $F(K_{n+1})$ are either in $\{0,1 \}$ or $\{0,-1 \}$ by \eqref{eq:4-1}, the number of elements in $\varphi(E)$ is exactly equal to the number of positive roots in $E$.  

Now suppose that the root system $F(K_{n+1})$ decomposes into the union of two nontrivial components $E_i$ for $i=1,2$.  Then there are $m_i$-dimensional linear subspaces $V_i$ of $\R^n$ such that $E_i=F(K_{n+1})\cap V_i$ and $m_1+m_2=n$, where $m_i\ge 1$.  
Since the number of positive roots in $E_i$, denoted by $p_i$, is at most $2^{m_i}-1$ by the observation above, we have $$p_1+p_2\le (2^{m_1}-1)+(2^{m_2}-1)<2^n-1.$$  However, since $F(K_{n+1})=E_1\cup E_2$ and the number of positive roots in $F(K_{n+1})$ is $2^n-1$ as remarked before, we must have $2^n-1=p_1+p_2$.  This is a contradiction.  Therefore, $F(K_{n+1})$ must be irreducible if it forms a root system.  
\end{proof}

The following is our main theorem. 

\noindent
\begin{theo} \label{theo:4-1}
 {\it Let $G$ be a connected finite simple graph with more than two nodes. Then the set $F(G)$ of facet vectors of the graph associahedron associated to $G$ forms a root system if and only if $G$ is a cycle graph. Moreover, the root system associated to the cycle graph with $n+1$ nodes is of type $A_n$.}
\end{theo}

The rest of this paper is devoted to the proof of Theorem~\ref{theo:4-1}. 
We begin with the following lemma. 

\noindent
\begin{lem} \label{lem:4.1.1}
 {\it Let $C_{n+1}$ be the cycle graph with $n+1$ nodes. Then $F(C_{n+1})$ forms a root system of type $A_n$.}
\end{lem}

\begin{proof}
An element $I$ in the graphical building set $B(C_{n+1})$ different from the entire set $[n+1]$ is one of the following:\\
(I)\ \ \ $\{i, i+1,\dots, j\}$ where $1\le i\le j\le n$,\\
(I\hspace{-.1em}I)\ \ $\{i,i+1,\dots, n+1\}$ where $2\le i\le n+1$,\\
(I\hspace{-.1em}I\hspace{-.1em}I) $\{i,i+1,\dots, n+1, 1,\dots,j\}$ where $1\le j<i\le n+1$ and $i-j\ge 2$.\\
Therefore the facet vector of the facet corresponding to $I$ is respectively given by 
\[ \sum_{k=i}^j e_k,\qquad -\sum_{k=1}^{i-1}e_k,\qquad -\sum_{k=j+1}^{i-1}e_k \]
according to the cases (I), (I\hspace{-.1em}I), (I\hspace{-.1em}I\hspace{-.1em}I) above.  Hence 
\begin{equation} \label{eq:4-3}
F(C_{n+1})=\big\{\pm \sum_{k=i}^j e_k \mid 1\le i<j\le n\big\}.
\end{equation}
This set forms a root system of type $A_n$.  Indeed, an isomorphism from ${\Z}^n$ to the lattice 
\[ 
\left\{(x_1, \dots, x_{n+1}) \in {\Z}^{n+1} \mid x_1+ \dots  + x_{n+1} = 0 \right\} 
\]
sending $e_i$ to $e_i-e_{i+1}$ for $i=1,2,\dots,n$ maps $F(C_{n+1})$ to the standard root system $\Delta(A_n)$ of type $A_n$ in \eqref{eq:4-1-1}. 
\end{proof}

The following lemma is a key observation. 

\noindent
\begin{lem} \label{lem:4.1.2}
 {\it Let $G$ be a connected simple graph.  
 Then $F(G)$ is \emph{centrally symmetric}, which means that $\alpha\in F(G)$ if and only if $-\alpha\in F(G)$ (note that $F(G)$ is centrally symmetric if  $F(G)$ forms a root system) if and only if 
the following holds:}
 \begin{equation} \label{eq:4.1.1}
 I \in B(G)\ \  \Longrightarrow\ \ V(G) \setminus I \in B(G).
 \end{equation}
\end{lem}

\begin{proof}
Let $V(G)=[n+1]$ as before 
and let $I$ be an element in $B(G)$ and $\alpha_I$ be the facet vector of the facet of $P_G$ corresponding to $I$.  If we set $e_{n+1}:=-(e_1+\dots+e_n)$, then $\alpha_I=\sum_{i\in I}e_i$.  Since
$
\alpha_I+\sum_{i\in [n+1]\backslash I}e_i=\sum_{i\in [n+1]}e_i=0,
$
we obtain $-\alpha_I=\sum_{i\in [n+1]\backslash I}e_i$ and this implies the lemma.   
\end{proof}

Using Lemma~\ref{lem:4.1.2}, we prove the following. 

\noindent
\begin{lem} \label{lem:4.1.3}
 {\it Let $G$ be a connected finite simple graph.
 Then $B(G)$ satisfies} \rm{(\ref{eq:4.1.1})} {\it if and only if $G$ is a cycle graph or a complete graph.}
\end{lem}
\begin{proof}
 If $G$ is a cycle or complete graph, then $F(G)$ is centrally symmetric by \eqref{eq:4-1} or \eqref{eq:4-3} and hence $B(G)$ satisfies \rm{(\ref{eq:4.1.1})} by Lemma~\ref{lem:4.1.2}. So the \lq\lq if" part is proven.  
 
 We shall prove the \lq\lq only if" part, so we assume that $B(G)$ satisfies \eqref{eq:4.1.1}.  Suppose that $G$ is not a complete graph.  Then there are $i\not=j\in V(G)$ such that $\{i,j \}$ is not contained in $B(G)$. By \eqref{eq:4.1.1}, $V(G) \setminus \{i,j \}$ is not contained in $B(G)$, which means that the induced subgraph $G|(V(G) \setminus \{i,j \})$ is not connected.  On the other hand, since $B(G)$ contains $\{i\}$ and $\{j\}$, $B(G)$ contains $V(G) \setminus \{i\}$ and $V(G) \setminus \{j\}$ by \eqref{eq:4.1.1}.
Hence
\begin{eqnarray} \label{eq:4.1.2}
\quad  &G|(V(G) \setminus \{i\}) \mathrm{\ and\ }\ G|(V(G) \setminus \{j\}) \\
    &\mathrm{are\ connected}. \nonumber
\end{eqnarray}

Let $k$ be the number of connected components of $G|(V(G) \setminus \{i,j \})$ and we denote its $k$ components by $G_1, \dots, G_k$  (Figure~\ref{fig:general_case}). By (\ref{eq:4.1.2}), the nodes $i$ and $j$ are respectively joined to every connected component by at least one edge. Since $G|(V(G_1) \cup \{i,j \})$ is connected, $G|(V(G_2)\cup \dots \cup V(G_k))$ is also connected by (\ref{eq:4.1.1}). However, $G|(V(G_2)\cup \dots \cup V(G_k))$ is the disjoint union of the connected subgraphs $G_2,\dots, G_k$. Therefore we have $k=2$.

 \begin{figure}[htbp]
 \begin{center}
  \includegraphics[width=5cm]{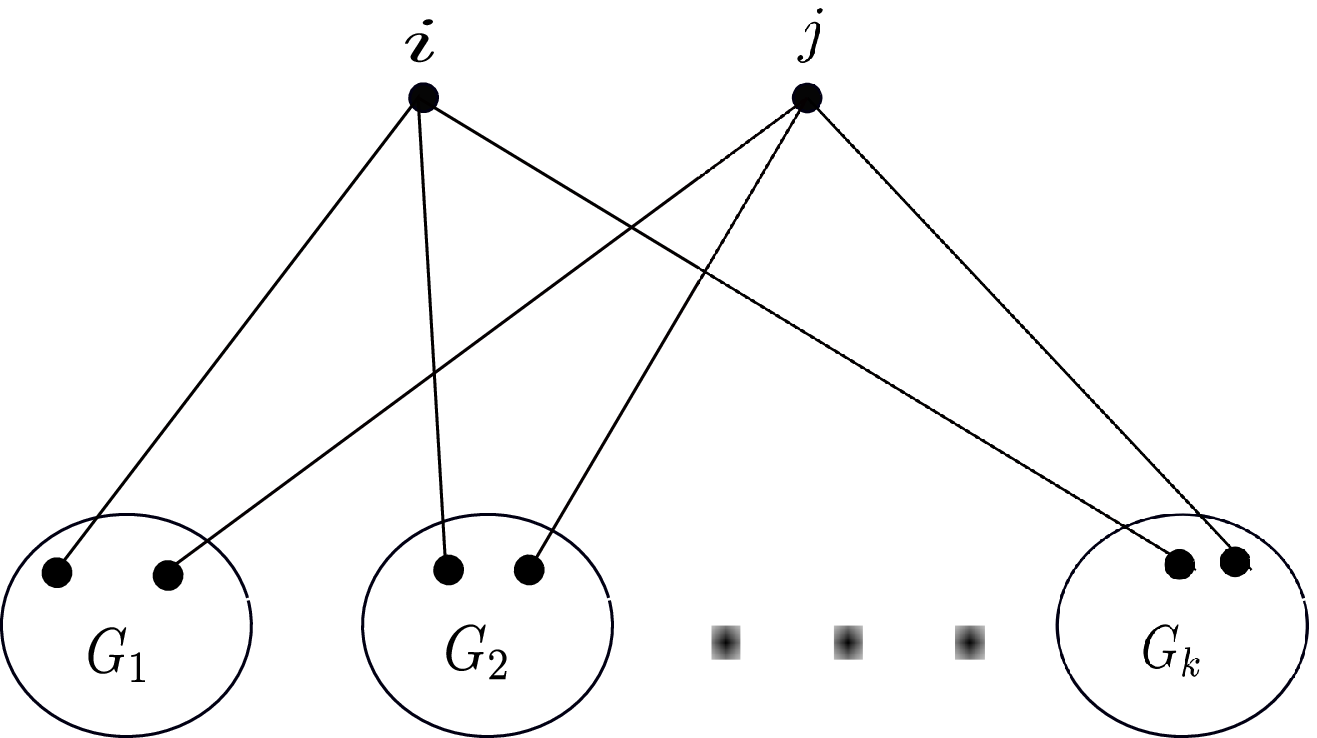}
 \end{center}
\caption{}  
\label{fig:general_case}
 \end{figure}
 
If $G_1$ and $G_2$ are both path graphs and the node $i$ is joined to one end node of $G_1, G_2$ respectively and the node $j$ is joined to the other end node of $G_1, G_2$, then $G$ is a cycle graph (Figure~\ref{fig:cycle_graph}).
 \begin{figure}[htbp]           
  \begin{center}
   \includegraphics[width=5cm]{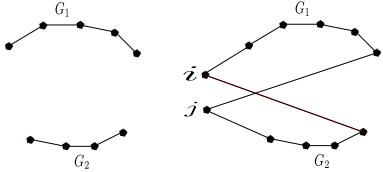}
  \end{center}
  \caption{the case of cycle graph}
  \label{fig:cycle_graph}
 \end{figure}
 
 We consider the other case, that is, either
\begin{enumerate}
\item[(I)] $G_1$ or $G_2$ is not a path graph, or 
\item[(II)] both $G_1$ and $G_2$ are path graphs but the nodes $i$ and $j$ are not joined to the end points of $G_1$ and $G_2$ (see Figure~\ref{fig:another_case}, left).  
\end{enumerate}
Then there exist nodes $i_1, j_1\in V(G_1)$ and $i_2, j_2\in V(G_2)$ such that 
\begin{enumerate}
\item[$\bullet$] $i_1$ and $i_2$ are joined to $i$,
\item[$\bullet$] $j_1$ and $j_2$ are joined to $j$, and
\item[$\bullet$] either the shortest path $P_1$ from $i_1$ to $j_1$ in $G_1$ is not the entire $G_1$ or the shortest path $P_2$ from $i_2$ to $j_2$ in $G_2$ is not the entire $G_2$.
\end{enumerate} 
\begin{figure}[htbp]           
  \begin{center}
   \includegraphics[width=6cm]{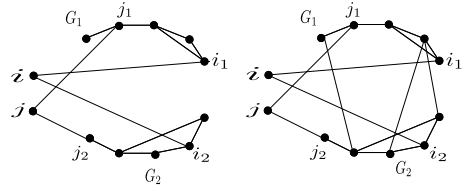}
  \end{center}
  \caption{the other case}
  \label{fig:another_case}
 \end{figure}
 
Without loss of generality we may assume that $P_1 \neq G_1$.  
Since $G|(V(P_1)\cup\{i,j,i_2,j_2\})$ is connected, so is $G|(V(G) \setminus (V(P_1)\cup\{i,j,i_2,j_2\}))$ by (\ref{eq:4.1.1}). This means that there is at least one edge joining $G_1$ and $G_2$ (Figure~\ref{fig:another_case}, right), and hence $G|(V(G) \setminus \{i,j \})$ is connected.  This contradicts that $G|(V(G) \setminus \{i,j \})$ consists of two connected components. 
\end{proof}

Now Theorem~\ref{theo:4-1} follows from  Lemmas~\ref{lem:4.1.0}, \ref{lem:4.1.1}, \ref{lem:4.1.2} and \ref{lem:4.1.3}.

\textbf{Acknowledgements.}\ \ \ 
I would like to thank Mikiya Masuda and Hiraku Abe for many interesting and fruitful discussions on this subject. 
This work was partially supported by Grant-in-Aid for JSPS Fellows 27 $\cdot$ 0184.

\end{document}